\documentclass[a4paper]{amsart}
\usepackage[latin1]{inputenc}
\usepackage{amssymb}
\usepackage{amsmath}
\usepackage{mathrsfs}
\usepackage{eufrak}
\usepackage{amsthm}
\usepackage{amsfonts}
\usepackage{textcomp}
\usepackage{graphicx}
\usepackage[pdftex]{color}
\usepackage{paralist}
\usepackage[shortlabels]{enumitem}
\usepackage{hyperref}
\usepackage{comment}
\usepackage[arrow, matrix, curve]{xy}

\newtheorem*{corollary*}{Corollary}
\newtheorem{theorem}{Theorem}[section]

\newtheorem{corollary}[theorem]{Corollary}
\newtheorem{lemma}[theorem]{Lemma}
\newtheorem{proposition}[theorem]{Proposition}

\newtheorem*{claim*}{Claim}

\theoremstyle{definition}

\newtheorem*{theorem }{Theorem}
\newtheorem*{theorem*}{Theorem}

\newtheorem{example}[theorem]{Example}

\theoremstyle{remark}

\numberwithin{equation}{theorem}

\makeatletter
\renewcommand*\env@matrix[1][\
arraystretch]{%
  \edef\arraystretch{#1}%
  \hskip -\arraycolsep
  \let\@ifnextchar\new@ifnextchar
  \array{*\c@MaxMatrixCols c}}
\makeatother

\begin{document}

\title{On monomial algebras having the double centraliser property}
\date{\today}



\author{Ren\'{e} Marczinzik}
\address{Institute of algebra and number theory, University of Stuttgart, Pfaffenwaldring 57, 70569 Stuttgart, Germany}
\email{marczire@mathematik.uni-stuttgart.de}

\begin{abstract}
Let $A$ be a finite dimensional algebra having the double centraliser property with respect to a minimal faithful projective-injective left module $Af$ for some idempotent $f$. We prove that in this case $A$ is a monomial algebra if and only if $A$ is a Nakayama algebra given by quiver and relations.
\end{abstract}

\maketitle
\section*{Introduction}
We assume throughout this article that algebras are finite dimensional algebras over a field $K$. A projective left module $Af$ with an idempotent $f$ is said to be \emph{minimal faithful projective-injective} in case $Af$ is a faithful projective-injective left module and every faithful projective-injective left module has $Af$ as a direct summand. The notion of a minimal faithful projective-injective right module is defined similarly. $A$ is said to have the \emph{double centraliser property} with respect to the minimal faithful projective-injective module $Af$ in case $A \cong End_{fAf}(Af)$ and $fAf \cong End_A(Af)$ (where the second isomorphism is automatic and always holds for any idempotent $f$).
This double centraliser property occurs in many places in mathematics, we mention the Schur-Weyl duality where $A$ is the Schur algebra $S(n,r)$ for $n \geq r$ and $fAf$ is the group algebra of the corresponding symmetric group and also the doublce centraliser property of blocks of category $\mathcal{O}$ arising from Lie theory where $A$ is some block of the Bernstein-Gelfand-Gelfang category $\mathcal{O}$ and $fAf$ is a symmetric local algebra, we refer to \cite{KSX} for proofs and more on this. Let $A$ be an algebra with minimal injective coresolution $(I_i)$ of the regular module $A$:
$$0 \rightarrow A \rightarrow I_0 \rightarrow I_1 \rightarrow \cdots .$$
The \emph{dominant dimension} of $A$ is defined as zero in case $I_0$ is not projective and equal to $\sup \{ n \geq 0 | I_i$ is projective for $i=0,1,...,n \}+1$ in case $I_0$ is projective. It can be shown that $A$ has dominant dimension at least one if and only if there is a minimal faithful projective-injective right module $eA$ if and only if there is a minimal faithful projective-injective left module $Af$, see for example chapter 4 of \cite{Ta}. Note that for algebras with dominant dimension at least one, one has $eAe \cong fAf$ as algebras.
Furthermore $A$ has the double centraliser property with respect to a minimal faithful projective-injective left module $Af$ if and only if it has dominant dimension at least two, see for example \cite{Ta} chapter 10. The class of algebras having dominant dimension at least two is very large and includes for example the higher Auslander algebras introduced by Iyama in \cite{Iya} and the Morita algebras introduced by Kerner and Yamagata recently in \cite{KerYam}.
Because of this equivalent characterisation via the dominant dimension of algebras having the double centraliser property with respect to a minimal faithful projective-injective module, we will often speak for short about algebras with dominant dimension at least two instead of the longer term of algebras having the double centraliser property with respect to a minimal faithful projective-injective module.
Recall that an algebra $A$ is called \emph{monomial} in case $A \cong KQ/I$ for some finite quiver $Q$ with an admissible ideal $I$ that is monomial, which means that it is generated by non-zero paths. Note that an algebra $A$ of the form $KQ/I_1$ with admissible non-monomial ideal $I_1$ can be isomorphic to an algebra $KQ/I_2$, where $I_2$ is an admissible monomial ideal and thus $A$ is a monomial algebra even though $I_1$ is not monomial, see for example exercise 4 in chapter I. of \cite{SkoYam}.
In this article we give a classification of monomial algebras having the double centraliser property with respect to a minimal faithful projective-injective module. Since we will deal here mainly with monomial algebras, we will assume in the following that every algebra is connected and given by quiver and admissible relations if not stated otherwise. Recall that a \emph{Nakayama algebra} (some authors call those algebras serial algebras) is an algebra where every indecomposable projective and every indecomposable injective module is uniserial. It can be shown that for Nakayama algebras in fact every indecomposable module is uniserial and the quiver of a Nakayama algebra can have only two shapes as in the following.
The quiver of a Nakayama algebra with a cyclic quiver:
$$Q=\begin{xymatrix}{ &  \circ^0 \ar[r] & \circ^1 \ar[dr] &   \\
\circ^{n-1} \ar[ur] &     &     & \circ^2 \ar[d] \\
\circ^{n-2} \ar[u] &  &  & \circ^3 \ar[dl] \\
   & \circ^5 \ar @{} [ul] |{\ddots} & \circ^4 \ar[l] &  }\end{xymatrix}$$
\newline
\newline
The quiver of an Nakayama algebra with a linear quiver:
$$Q=\begin{xymatrix}{ \circ^0 \ar[r] & \circ^1 \ar[r] & \circ^2 \ar @{} [r] |{\cdots} & \circ^{n-2} \ar[r] & \circ^{n-1}}\end{xymatrix}$$

Especially: Every Nakayama algebra is a monomial algebra. We refer for example to \cite{AnFul} and \cite{SkoYam} for proofs and more on Nakayama algebras.
For a module $N$, $add(N)$ denotes the full subcategory of $mod-A$ consisting of finite direct sums of indecomposable modules that are direct summands of $N$. 
Recall that a module $N$ is a generator in case it contains every indecomposable projective module as a direct summand and $N$ is a cogenerator in case it contains every indecomposable injective module as a direct summand. We call a module $N$ \emph{generator-cogenerator} in case it is a generator and a cogenerator. A module $N$ is called basic in case it does not have a direct summand of the form $M^2$ where $M$ is an indecomposable non-zero module. $D:=Hom_K(-,K)$ denotes the natural duality of a finite dimensional algebra.
Our main theorem of this article can be stated as follows:
\begin{theorem*}
Let $A$ be a finite dimensional algebra.
The following are equivalent:
\begin{enumerate}
\item $A$ is a monomial algebra with dominant dimension at least two.
\item $A \cong End_B(M)$, where $B$ is a Nakayama algebra and $M$ a basic generator-cogenerator in \newline $add(B \oplus D(B) \oplus D(B)/soc(D(B)))$.
\item $A$ is a Nakayama algebra with dominant dimension at least two.
\end{enumerate}
\end{theorem*}
We apply this theorem to give also a classification of monomial Morita algebras.
The author thanks Aaron Chan for useful discussions and allowing him to use his example of a monomial algebra with dominant dimension equal to one that is not a Nakayama algebra in \ref{finalexample}.

\section{Proof of the main theorem}
We assume that all algebras are given by quiver and relations and are connected finite dimensional algebras over a field $K$. We assume that all modules are finite dimensional right if not stated otherwise. We remark that we still often use left modules when talking about minimal faithful projective-injective left modules, since here the double centraliser property has a nicer form when using left instead of right modules.
We assume that the reader is familiar with the basics of representation theory of finite dimensional algebras and refer for example to the books \cite{ASS}, \cite{SkoYam} and \cite{DW}. We refer also to \cite{Yam} for a survey article that treats dominant dimension and \cite{Ta} for a textbook that treats dominant dimension and double centraliser properties.
Before we can prove the main theorem of this article we recall some results from the literature.
\begin{theorem} \label{domdim2chara}
The following are equivalent for a finite dimensional algebra $A$:
\begin{enumerate}
\item $A$ has dominant dimension at least two.
\item $A \cong End_B(M)$ for an algebra $B$ with a generator-cogenerator $M$.
\item $A$ has the double centraliser property with respect to a minimal faithful projective-injective left module $Af$.
\end{enumerate}
\end{theorem}
\begin{proof}
See for example \cite{Ta} chapter 10 or \cite{Rin}.
\end{proof}
The algebra $B$ as in the previous theorem is called the \emph{base algebra} of an algebra $A$ of dominant dimension at least one and is uniquely determined as the algebra $fAf$ when $Af$ is the minimal faithful projective-injective left $A$-module.
\begin{theorem} \label{yamagatatheorem}
Let $A$ be a Nakayama algebra and $M$ a basic generator-cogenerator.
Then $End_A(M)$ is a Nakayama algebra if and only if $M \in add(B \oplus D(B) \oplus D(B)/soc(D(B)))$.
\end{theorem}
\begin{proof}
This is the main result of \cite{Yam2} specialised to quiver algebras and generator-cogenerators.
\end{proof}
\begin{proposition} \label{nakayamapropo}
Let $A$ be a monomial algebra with minimal faithful projective-injective module left module $Af$. Then $fAf$ is a Nakayama algebra.
\end{proposition}
\begin{proof}
See \cite{Mar}, proposition 2.19.
\end{proof}
\begin{lemma} \label{APTlemma}
Let $A$ be an algebra with a basic generator-cogenerator $M$ and let $B=End_A(M)$.
Let $M_i$ be the indecomposable direct summands of $M$.
The indecomposable projective $B$-modules are exactly $Hom_A(M,M_i)$ and the indecomposable projective-injective $B$-modules are exactly $Hom_A(M,M_i)$ for the injective indecomposable $A$-modules $M_i$.
\end{lemma}
\begin{proof}
This is a special case of lemma 3.1. in \cite{APT}.
\end{proof}
As a generalisation of quasi-Frobenius algebras, \emph{QF-2 algebras} were defined as algebras such that the socle of every indecomposable projective module is simple. We refer to \cite{Yam} for a more on those algebras.

\begin{proposition} \label{nakayamaqf2}
Let $A$ be a Nakayama algebra with a basic generator-cogenerator $M$.
Then $B=End_A(M)$ is a QF-2 algebra.
\end{proposition}
\begin{proof}
Recall from \ref{APTlemma} that every indecomposable projective $B$-module is isomorphic to $Hom_A(M,M_i)$ when $M_i$ denote the indecomposable direct summands of $M$.
Let $I$ be an indecomposable injective direct summand of $M$. Then by \ref{APTlemma} the $B$-modules $Hom_A(M,I)$ are projective-injective indecomposable and thus have simple socle.
Now let $N$ be an indecomposable direct summand of $M$ that is not injective.
Since $A$ is a Nakayama algebra any indecomposable $A$-module is uniserial and thus has simple socle. Now since $N$ has simple socle, its injective envelope $I(N)$ is indecomposable and there is the following short exact sequence:
$$0 \rightarrow N \rightarrow I(N) \rightarrow K \rightarrow 0,$$
where $K$ is the cokernel of the inclusion $N \rightarrow I(N)$.
Applying the functor $Hom_A(M,-)$ to this short exact sequence and using that it is left exact we obtain an inclusion of $B$-modules:
$$0 \rightarrow Hom_A(M,N) \rightarrow Hom_A(M,I(N)).$$
This shows that the indecomposable projective $B$-module $Hom_A(M,N)$ is a submodule of the indecomposable projective-injective $B$-module $Hom_A(M,I(N))$.
But with $Hom_A(M,I(N))$ also every of its submodules has simple socle and thus $Hom_A(M,N)$ has simple socle, which finishes the proof.
\end{proof}

\begin{lemma} \label{monomialqf2impliesnaka}
Let $A$ be a monomial QF-2 algebra.
Then $A$ is a Nakayama algebra.
\end{lemma}
\begin{proof}
Assume $A$ is a monomial QF-2 algebra but not a Nakayama algebra. We can assume that $A \cong KQ/I$ with an admissible monomial ideal $I$. We will show that this gives a contradiction. Since $A$ is not a Nakayama algebra, there is a point in the quiver of $A$ such that at this point there start at least two arrows or there end at least two arrows.
We look at both cases. \newline
\underline{Case 1:} Assume there is a point $i$ in the quiver of $A$ where at least two arrows start. Since we assume $A$ to be QF-2, the indeocomposable projective module $e_i A$ has simple socle, which is equivalent to the condition that here is a unique longest path starting at $i$ since $A$ is assumed to be monomial.
But since $A$ is monomial and there start at least two arrows $\alpha_1$ and $\alpha_2$ at $i$ there are at least two longest paths $p_1= \alpha_1 \cdots$ and $p_2 = \alpha_2 \cdots$ starting with $\alpha_1$ and $\alpha_2$ respectively. Since $A$ is monomial and the admissible ideal $I$ contains no commutativity relations, the paths $p_1$ and $p_2$ can not be identified and $e_i A$ can not have simple socle. This is a contradiciton. \newline
\underline{Case 2:} Assume there is a point $i$ in the quiver of $A$ where at least two arrows end and assume $A$ is not a Nakayama algebra. In this case look at the algebra $B:=A^{op}$, the opposite algebra of $A$. Then $B$ is a monomial QF-2 algebra that is not a Nakayama algebra with a point $i$ where at least two arrows start and we are in case 1 and obtain a contradiction.
\end{proof}

With all the work done, we can now give an easy proof of our main theorem.
\begin{theorem} \label{mainresult}
Let $A$ be a finite dimensional algebra.
The following are equivalent:
\begin{enumerate}
\item $A$ is a monomial algebra with dominant dimension at least two.
\item $A$ is a Nakayama algebra with dominant dimension at least two.
\item $A \cong End_B(M)$, where $B$ is a Nakayama algebra and $M$ a basic generator-cogenerator in $add(B \oplus D(B) \oplus D(B)/soc(D(B)))$.

\end{enumerate}
\end{theorem}
\begin{proof}
We first show (1) $\implies$ (2): Assume $A$ is a monomial algebra with dominant dimension at least two. Thus there is a minimal faithful projective-injective  left $A$-module $Af$ such that $A \cong End_{fAf}(Af)$ by \ref{domdim2chara}. By \ref{nakayamapropo} the algebra $fAf$ is a Nakayama algebra and by \ref{nakayamaqf2} $A$ is a QF-2 algebra since $Af$ is a generator-cogenerator. By \ref{monomialqf2impliesnaka} $A$ is then a Nakayama algebra. \newline
That (2) implies (3) follows directy from \ref{domdim2chara} combined with \ref{yamagatatheorem}.
Assume (3) holds, then by \ref{yamagatatheorem} $A$ is a Nakayama algebra with dominant dimension at least two and thus also a monomial algebra with dominant dimension at least two since every Nakayama algebra is a monomial algebra and thus (1) follows.

\end{proof}
Following \cite{KerYam}, a \emph{Morita algebra} is by definition an algebra $A$ with dominant dimension at least two such that $fAf$ is selfinjective algebra when $Af$ denotes the minimal faithful projective-injective left module.
As a corollary of our main result, we can give a classification of monomial Morita algebras. Note that selfinjective Nakayama algebras are exactly those Nakayama algebras where the indecomposable projective modules all have the same vector space dimension, see for example \cite{SkoYam} theorem 6.15. in chapter IV.
\begin{corollary}
Let $A$ be a monomial Morita algebra. Then $A$ is isomorphic to $End_B(M)$, where $B$ is a selfinjective Nakayama algebra and $M=B \oplus N$, where $N$ is a direct sum of distinct indecomposable modules of the form $P/soc(P)$ when $P$ is an indecomposable projective $B$-module.
\end{corollary}
\begin{proof}
This is a direct consequence of \ref{mainresult} when noting that $D(B) \cong B$ and thus \newline $add(B \oplus D(B) \oplus D(B)/soc(D(B)))=add(B \oplus B/soc(B))$ because $B$ has to be selfinjective.
\end{proof}
We note that Nakayama algebras with dominant dimension at least two are an interesting class of algebras, which are characterised in \cite{Ful} lemma 4.3 and in chapter 5 of \cite{NRTZ} those algebras are characterised using tilting theory.
While all monomial algbras with dominant dimension at least two are Nakayama algebras, there are monomial algebras with dominant dimension equal to one that are not Nakayama algebras as the following example due to Aaron Chan shows:
\begin{example} \label{finalexample}
Let $Q$ be the quiver:
$$Q=\begin{xymatrix}{ \circ^1 \ar[r]^{\alpha_1} & \circ^2  \ar[dr]^{\alpha_4} \ar[dl]^{\alpha_3} & \circ^3 \ar[l]^{\alpha_2} \\ \circ^4 & & \circ^5}\end{xymatrix}$$
and let $I$ be the admissible ideal $I=<\alpha_1 \alpha_3 , \alpha_2 \alpha_4 >$.
Let $A=kQ/I$. Then $A$ is a monomial algebra with dominant dimension equal to one, but $A$ is not a Nakayama algebra.

\end{example}

\end{document}